\documentclass[a4paper,12pt]{article}     
\usepackage{amsmath,amscd,amssymb}        
\usepackage{latexsym}                     
\usepackage[english, german]{babel}                

\usepackage{theorem}                 
\usepackage{color}

\newtheorem{theorem}{Theorem}
\newtheorem{corollary}{Corollary}

\theorembodyfont{\rmfamily}

\newenvironment{definition}
{\smallskip\noindent{\bf Definition\/}:}{\smallskip\par}

\newenvironment{remark}
{\smallskip\noindent{\bf Remark\/}.}{\smallskip\par}

\newenvironment{proof}{\begin{ProofwCaption}{Proof}}{\end{ProofwCaption}}
\newenvironment{proof*}[1]{\begin{ProofwCaption}{{#1}}}{\end{ProofwCaption}}
\newenvironment{ProofwCaption}[1]%
 {\addvspace\theorempreskipamount \noindent{\it #1.}\rm}%
 {\qed \par \addvspace\theorempostskipamount}
\newcommand{\qedsymbol}{\mbox{$\Box$}}
\newcommand{\qed}{\hfill\qedsymbol}

\newcommand{\CC}{{\mathbb C}}

\newcommand{\Aut}{{\rm Aut}\,}

\title{Klein foams as families of real forms of Riemann surfaces}
\author{Sabir M.~Gusein-Zade, Sergey M.~Natanzon
\thanks{
AMS 2010 Math. Subject Classification: 30F50, 57M20, 14H30, 14H37.
}}
\date{}

\begin{document}
\selectlanguage{english}

\maketitle

\medskip
\noindent Moscow State Lomonosov University, Faculty of mechanics and mathematics,\\
Moscow, GSP-1, 119991, Russia\\
E-mail: sabir@mccme.ru

\medskip
\noindent National Research University Higher School of Economics, 20 Myasnitskaya Ulitsa, Moscow 101000, Russia\\
ITEP 25 B.Cheremushkinskaya, Moscow 117218, Russia ,\\
 e-mail: natanzons@mail.ru

\begin{abstract}
Klein foams are analogues of Riemann surfaces for surfaces with one-dimensional singularities.
They first appeared in mathematical physics (string theory etc.).
By definition a Klein foam is constructed from Klein surfaces by gluing segments on their boundaries.
We show that, a Klein foam is equivalent to a family  of real forms of a complex algebraic  curve
with some structures.
This correspondence reduces investigations of Klein foams to investigations of real forms of Riemann surfaces.
We use known properties of real forms of Riemann surfaces to describe some topological and analytic properties of Klein foams.
\end{abstract}

\section{Introduction}
A Klein foam (a seamed surface) is constructed from Klein surfaces by gluing segments on their boundaries.
Klein foams first appeared in mathematical physics (string theory, A-models) \cite{B,KR, R}
(see also \cite{MMN}) and have applications in mathematics: \cite{MV,AN1,AN3} etc.

A Klein surface is an analog of a Riemann surface for surfaces with boundaries and non-orientable ones \cite{AG,N1}.
One can consider a Klein foam as an analog of a Riemann surface for surfaces with one-dimensional singularities.
Further on we consider only compact Riemann and Klein surfaces.

Following \cite{CGN} we also require existence of a dianalytic map from a Klein foam to the complex disc.
According to \cite{CGN} such function exists if and only if the corresponding topological foam is strongly oriented.
According to \cite{N} this condition makes it possible to extend the 2D topological field theory with Klein surfaces \cite{AN}
to Klein foams.

\vspace{1ex}
In this paper we prove, that a Klein foam is a collection of real forms of a complex algebraic curve.
Any Klein surface is the quotient $S/\tau$, where $\tau :S\rightarrow S$ is an antiholomorphic involution of a Riemann surface $S$.
In terms of algebraic geometry the pair $(S,\tau)$ is a complex algebraic curve with an involution of complex conjugation.
Thus it is a real algebraic curve \cite{AG}. We will say also that $\tau$ is a real form of $S$. The fixed points of $\tau$
correspond to the boundary $\partial(S/\tau)$ of the quotient $S/\tau$ and to real points of the corresponding
real algebraic curve.
The fixed points form closed contours that are called \textit{ovals} \cite{N4}. In what follows we consider only real algebraic
curves with real points, that is with non-empty $\partial(S/\tau)$.

A general Riemann surface has not more that one antiholomorphic involution. However, there exist Riemann surfaces
with several antiholomorphic involutions \cite{Na1,Na2,Na3}. Here we prove that the Klein foams are in one-to-one correspondence
with the equivalence classes of the families of real forms of Riemann surfaces, i.e. of the collections
$\{\widehat{S}, G, (G_1,\widehat{\tau}_1),\ldots, (G_r,\widehat{\tau}_r) \}$ consisting of:
\newline
1) a compact Riemann surface $\widehat{S}$;
\newline
2) a finite subgroup $G$ of the group $\Aut(\widehat{S})$ of holomorphic automorphisms of $\widehat{S}$;
\newline
3) real forms $\widehat{\tau}_1,\ldots, \widehat{\tau}_r$ of $\widehat{S}$ such that $\widehat{\tau}_i G\widehat{\tau}_i=G$;
\newline
4) subgroups $G_i\subset G$ ($i=1,2,\ldots, r$) generating $G$ and such that $\widehat{\tau}_i G_i\widehat{\tau}_i=G_i$.


A collection of real forms of a Riemann surface has some non-trivial ``collective'' properties.
Initially these properties were observed in \cite{Na5,Na6,Na7,Na4}. These results were developed and refined
in a long series of publications (see \cite{BCGG} and references therein). We use these properties of real forms
to describe some non-trivial combinatorial properties of analytic structures of Klein foams. In particular,
we  give bounds for the number of non-isomorphic Klein surfaces in a Klein foam and for the total number of ovals in a Klein foam.

\section{Klein foams}

A Klein foam is a topological foam with an analytic structure. A topological foam is obtained from surfaces with boundaries
by gluing some segments on the boundaries \cite{R,AN1}. For our goal we modify this definition considering instead of a surface
with boundary a surface without boundary, but with an involution such that the corresponding
quotient is homeomorphic to the first one. This change of the definition does not change the notion
itself.

A {\em generalized graph} is a one-dimensional space which consist of finitely many vertices and edges,
where edges are either segments connecting {\bf different} vertices or isolated circles without vertices on them.
A pair of vertices may be connected by several edges.

\vspace{1ex}

A {\em normal topological foam} $\Omega$ is a triple $(S,\Delta,\varphi)$, where
\vspace{-1ex}
\begin{itemize}
\item $S = S(\Omega)$ is a closed (i.e. compact without boundary, but usually disconnected) oriented surface
with a reversing the orientation involution with the fixed point set being a closed curve $L\neq\emptyset$,
that is finite disconnected union of simple contours;
\item $\Delta=\Delta(\Omega)$ is a generalized graph;
\item $\varphi=\varphi_{\Omega}: L\to\Delta$ is {\em the gluing map}, that is, a map such that:
\newline (a) $\mbox{Im\,}\varphi=\Delta$;
\newline (b) on each connected component of $L$, $\varphi$ is a homeomorphism on a circle in $\Delta$;
\newline (c) for an edge $l$ of $\Delta$, any connected component of $S$ contains at most one connected
component of $\varphi^{-1}(l\setminus\partial l)$;
\newline (d) (the normality condition) for $\check{\Omega}=S\cup_{\varphi}\Delta$ (the result of the gluing
of $S$ along $\Delta$) and for each vertex $v$ from the set $\Omega_b$ of vertices of the graph $\Delta$,
its punctured neighbourhood in $\check{\Omega}$ is connected.
\end{itemize}

A triple $(S,\Delta,\varphi)$ which satisfies all the properties above but (c) will be called a {\em topological pseudofoam}.
The same terminology will be applied to analytic and Klein foams defined below.

Normal topological foams arise naturally in the theory of Hurwitz numbers \cite{AN1,AN3}.
A foam $\Omega=(S,\Delta,\varphi)$ will be called {\em connected} if $\check{\Omega}$ is connected.
In what follows topological foams are assumed to be normal and connected. Let $\Omega_{b}$
be the set of vertices of the graph $\Delta$.

A \textit{morphism} $f$ of topological foams $\Omega^{\prime}\rightarrow \Omega^{\prime\prime}$
($\Omega^{\prime}=(S^{\prime}, \Delta^{\prime}, \varphi^{\prime})$,
$\Omega^{\prime\prime}=(S^{\prime\prime},\Delta^{\prime\prime}, \varphi^{\prime\prime})$) is a pair $(f_{S},f_{\Delta})$
of (continuous) maps $f_{S}:S^{\prime}\to S^{\prime \prime}$ and
$f_{\Delta}: \Delta^{\prime}\to\Delta^{\prime\prime}$ such that $f_{S}$ is an orientation preserving
ramified covering commuting with the involutions on $S^{\prime}$ and
$S^{\prime\prime}$, $\varphi^{\prime\prime}\circ f_{S} = f_{\Delta}\circ \varphi^{\prime}$ and
$f_{\Delta}|_{\Delta^{\prime}\setminus\Omega^{\prime }_{b}}$ is a
local homeomorphism $\Delta^{\prime}\setminus\Omega^{\prime}_{b}$
on $\Delta^{\prime\prime}\setminus\Omega^{\prime\prime}_{b}$.

\vspace{1ex}

An {\em analytic foam} is a topological foam $\Omega=(S,\Delta,\varphi)$, where $S$ is a compact Riemann surface,
the involution of which is antiholomorphic. A morphism $f$ of analytic foams
$\Omega^{\prime}\rightarrow \Omega^{\prime\prime}$ ($\Omega^{\prime}=(S^{\prime}, \Delta^{\prime}, \varphi^{\prime})$,
$\Omega^{\prime\prime}=(S^{\prime\prime},\Delta^{\prime\prime}, \varphi^{\prime\prime})$) is a morphism
$(f_{S},f_{\Delta})$ of the corresponding topological foams such that $f_{S}$ is complex analytic.

The simplest analytic foam is $\Omega_{0}=(\overline{\mathbb{C}}, S^1, Id)$, where $\overline{\mathbb{C}}=\mathbb{C}\cup\infty$
is the Riemann sphere with the involution $z\mapsto \bar{z}$. An \textit{analytic function} on an analytic foam $\Omega$
is a morphism of $\Omega$ to $\Omega_{0}$.

A \textit{Klein foam} is an analytic foam $\Omega=(S,\Delta,\varphi)$ admitting an everywhere locally non-constant
analytic function.
This condition appeared first in \cite{CGN} and is equivalent to the condition of being ``strongly oriented''
for the corresponding topological foam.
The category of strongly oriented topological foams \cite{N} is a subcategory of topological foams that allows
the topological field theory to be extended to the Klein topological field theory \cite{AN}.

A Klein (pseudo)foam $\Omega$ will be called \textit{compressed} if the coincidence of $F(x)$ and $F(x')$
for two points $x$ and $x'$ of $\Delta_{\Omega}$ and for all analytic functions $F$ on $\Omega$
implies that $x=x'$. For any Klein foam $\Omega$ there exists a unique compressed Klein pseudofoam $\Omega'$
with a morphism $\Psi:\Omega\rightarrow \Omega'$ such that $\Psi_S$ is an isomorphism. This follows from the following
fact. In \cite[Theorem 2.1]{CGN}, it was shown that there exists a connected Riemann surface $\check{S}$ with
an antiholomorphic involution $\check{\tau}$ and an analytic nowhere locally constant maps $f:\check{\Omega}\to \check{S}$,
commuting with the involutions on $S$ and on $\check{S}$, such that the corresponding morphism of the foams establishes
an isomorphism between the fields of analytic functions on ${\Omega}$ and on $\check{S}$ respectively.
One can see that $\Omega'$ is obtained from $\Omega$ by gluing all the points $x$ and $x'$ of $\Delta(\Omega)$
with $f(x)=f(x')$. (In particular the graph $\Delta_{\Omega'}$ coincides with the set of the real points
of $\check{S}$ cut into edges by the ramification points of $f$.) We shall say that $\Omega'$ is the \textit{compressing}
of $\Omega$.

Two Klein foams $\Omega_1$ and $\Omega_2$ will be called \textit{weakly isomorphic} if their compressings are isomorphic.

\section{Real forms.}

\begin{definition}
 An {\em equipped family of real forms} of a Riemann surface $\widehat{S}$ is a collection
 $\{\widehat{S}, G, (G_1,\widehat{\tau}_1),\ldots, (G_r,\widehat{\tau}_r) \}$ consisting of:
 \begin{itemize}
  \item a compact Riemann surface $\widehat{S}$;
  \item antiholomorphic involutions (i.e. real forms) $\{\widehat{\tau}_1,\ldots, \widehat{\tau}_r\}$ of $\widehat{S}$;
  \item a finite subgroup $G$ of the group of holomorphic automorphisms of $\Aut(S)$ such that
  $\widehat{\tau}_i G\widehat{\tau}_i=G$ for each $i=1,2,\ldots, r$;
  \item subgroups $G_i\subset G$ such that $\widehat{\tau}_i G_i\widehat{\tau}_i=G_i$ for $i=1,2,\ldots, r$
  and $G$ as a group is generated by $G_1$, \dots, $G_r$;
  \item the inclusion $G_i\subset G$ generates a morphism of Klein surfaces $S/G_i\rightarrow S/G$ that is
  a homeomorphism on any oval of $S/G$.
 \end{itemize}
\end{definition}

\begin{definition}
 Two equipped families of real forms of Riemann surfaces
 $$
  \{\widehat{S}, G, (G_1,\widehat{\tau}_1),\ldots, (G_r,\widehat{\tau}_r)\}\quad \text{and} \quad
  \{\widehat{S}, G^\prime, (G_1^\prime,\widehat{\tau}_1^\prime),\ldots, (G_r^\prime,\widehat{\tau}_r^\prime)\}
 $$
 are equivalent if there exists $h\in\Aut{\widehat{S}}, h_i\in G, l_i\in G_i$, such that
 $$
 hGh^{-1}=G',\quad  hh_iG_ih_i^{-1}h^{-1}=G_i', \quad hh_il_i\widehat{\tau}_il_i^{-1}h_i^{-1}h^{-1}=\widehat{\tau}_i'\,.
 $$

Two equipped families of real forms of Riemann surfaces
$$
 \{\widehat{S}, G, (G_1,\widehat{\tau}_1),\ldots, (G_r,\widehat{\tau}_r)\}\quad \text{and} \quad
 \{\widehat{S}', G', (G_1^\prime,\widehat{\tau}_1^\prime),\ldots, (G_r^\prime,\widehat{\tau}_r^\prime)\}
$$
 are equivalent if there exists an  analytic isomorphism $H:\widehat{S}\to\widehat{S}^\prime$
 such that
$$
 \{\widehat{S}',HGH^{-1}, (HG_1H^{-1},H\widehat{\tau}_1H^{-1}),\ldots, (HG_rH^{-1},H\widehat{\tau}_rH^{-1})\}
$$
and
$$
 \{\widehat{S}', G^\prime (G_1^\prime,\widehat{\tau}_1^\prime),\ldots, (G_r^\prime,\widehat{\tau}_r^\prime)\}
$$
are equivalent.
\end{definition}


\section{Main theorem.}

\begin{theorem}\label{T1} There exists a natural  one-to-one
correspondence between the classes of weakly isomorphic Klein foams and the equivalence classes of equipped families
of real forms of Riemann surfaces.
\end{theorem}

\begin{proof} Consider a Klein foam $\Omega=(S,\Delta,\varphi)$ with $S$ consisting of connected components
$S_i$, $i=1,\ldots, r$. Let $\tau_i$ be the restriction of the (antiholomorphic) involution to $S_i$.

In \cite[Theorem 2.1]{CGN}, it was shown that there exists a connected Riemann surface $\check{S}$ with
an antiholomorphic involution $\check{\tau}$ and analytic nowhere locally constant maps $f:\check{\Omega}\to \check{S}$,
commuting with the involutions on $S$ and on $\check{S}$, such that the corresponding morphism of the foams establishes
an isomorphism between the fields of analytic functions on ${\Omega}$ and on $\check{S}$ respectively.
Consider the restrictions $f_i: S_i\to \check{S}$ of $f$ to $S_i$.

Let $\check{S}^\circ$ be the surface $\check{S}$ without all the critical values of the maps $f_i$
and $S^{\circ}= f^{-1}(\check{S}^\circ)$. Consider an uniformization $U\rightarrow U/\check{\Gamma}=\check{S}^\circ$.
(Excluding trivial cases we assume that $U$ is the upper half-plane in $\CC$ and $\Gamma$ is a Fuchsian group.)

The restriction $f^\circ_i:  S^\circ_i\to \check{S}^\circ$ of $f_i$ to $S^\circ_i= S_i\cap S^\circ$ is a covering
without ramification. Thus there exists a subgroup $\Gamma_i\subset\check{\Gamma}$, that uniformises
$S^\circ_i= U/\Gamma_i$ and generates the map $f^\circ_i$ by the inclusion $\Gamma_i\subset\check{\Gamma}$.
This subgroup $\Gamma_i$ is well-defined up to conjugation in $\Gamma$. Let $\sigma_i$
be a lifting to $U$ of the involution $\tau_i$ to $S^\circ_i$, such that
$\sigma_i\Gamma_i\sigma_i=\Gamma_i$. This $\sigma_i$ is well-defined up to conjugation in $\Gamma_i$.

The subgroup $\Gamma_i$ has a finite index in $\Gamma$ and there are only finitely many different
subgroups of $\check{\Gamma}$ conjugate to $\Gamma_i$. Let us consider the intersection of all
the subgroups $\Gamma_i$ and all their conjugates:
$$
\hat{\Gamma}=\bigcap_{j=1}^r\bigcap_{a\in\check{\Gamma}}a^{-1}\Gamma_j a\subset\check{\Gamma}\,.
$$

Since all the involutions $\tau_i$ give one and the same involution $\check{\tau}$ on the surface $\check{S}$,
one has $\sigma_i\sigma_j \in\check{\Gamma}$ for all $i$ and $j$. Thus $\sigma_ia\sigma_j\in \check{\Gamma}$
for any $a\in\check{\Gamma}$ and
\begin{eqnarray*}
\sigma_i\hat{\Gamma}\sigma_i&=&
\sigma_i\left(\bigcap_{j=1}^r\bigcap_{a\in\check{\Gamma}}a^{-1}\Gamma_j a\right)\sigma_i=\\
&{\ }&\bigcap_{j=1}^r\bigcap_{a\in\Gamma}
\sigma_ia^{-1}\sigma_j\Gamma_j \sigma_ja\sigma_i=
\bigcap_{j=1}^r\bigcap_{b\in\Gamma}b^{-1}\Gamma_j b=\hat{\Gamma}\,.
\end{eqnarray*}

Consider now the Riemann surface $\hat{S}^\circ = U/\hat{\Gamma}$. The involution $\sigma_i$ generates
an involution $\hat{\tau}^0_i$ on $\hat{S}^\circ$. The inclusion $\hat{\Gamma}\subset\check{\Gamma}$ generates a covering
$\Phi^\circ: \hat{S}^\circ\rightarrow \check{S}^\circ$. The inclusion $\check{S}^\circ\subset\check{S}$ generates a branching
covering of Riemann surfaces $\hat{S}\rightarrow\check{S}$. The involution $\hat{\tau}^0_i$ generates an antiholomorphic
involution $\hat{\tau}_i$ on $\hat{S}$. Moreover $\hat{\Gamma}$ is a normal subgroup of $\check{\Gamma}$ and thus the subgroup
$\check{\Gamma}/\hat{\Gamma}$ in a natural way is isomorphic to a subgroup $G\subset\Aut(\hat{S})$ of the group of biholomorphic
automorphisms of $\hat{S}$.

The subgroups $\Gamma_i$ from this construction also have an interpretation in terms of
the Riemann surface $\hat{S}$.
Let us consider the group $\Aut^*(\hat{S})$ of biholomorphic and antiholomorphic automorphisms of $\hat{S}$.
All the involutions $\hat{\tau}_i$ belong to $\Aut^*(\hat{S})$. Moreover, the involution $\sigma_i$ preserves
the subgroup $\Gamma_i\subset\check{\Gamma}$. The homomorphism $\phi: \check{\Gamma}/\hat{\Gamma} \rightarrow\Aut(\hat{S})$
maps the subgroup $\Gamma_i$ to a subgroup $G_i\subset G$, such that $\hat{\tau}_iG_i\hat{\tau}_i=G_i$.

Thus we have constructed an equipped family of real forms of a compact Riemann surface
corresponding to a Klein foam $\Omega=(S,\Delta,\varphi)$.
This family depends on the choice of the uniformizing subgroups $\Gamma_i$ and of the pull-backs $\sigma_i$.
This means that the family is defined up to equivalence.

Since the compressing of a real foam depends only on the maps $f_i$, the weak equivalence class of the foam $\Omega$
is determined only by the equivalence class of the corresponding equipped family
of real forms of a Riemann surface.
\end{proof}

\begin{remark}
 One can see that there are finitely many Klein foams corresponding to an equivalence class of
 equiped families of real forms of a Riemann surface. They are determined by some additional combinatorial data.
\end{remark}

\section{Some topological and analytical properties of Klein foams.}

Consider a Klein foam
$$
\Omega=(S,\Delta,\varphi), \quad \text{where} \quad   S=\coprod\limits_{i=1}^r(S_i,\tau_i)\,.
$$
Denote by $g_i$ and $k_i$ the genus of $S_i$ and the number of ovals of $S_i$
(i.e. the number of the connected components of $\partial(S_i/\tau_i)$) respectively.

We say that $(S_i,\tau_i)$ and $(S_j,\tau_j)$ are
\textit{foam-equivalent} if the images under $\varphi$ of the fixed point sets of $\tau_i$ and $\tau_j$ coincide.
Otherwise we say that $(S_i,\tau_i)$ and $(S_j,\tau_j)$ are \textit{foam-different}.

Recall that Klein surfaces $(S',\tau')$ and $(S'',\tau'')$ are isomorphic if there exists a biholomorphic map $h:S'\to S''$
such that $\tau''=h\tau'h^{-1}$. We say that a real algebraic curve $(S',\tau')$ \textit{covers} $(S'',\tau'')$
if there exists a holomorphic map $h:S'\to S''$ such that $\tau''h=h\tau'$. We say that Klein surfaces $(S',\tau')$
and $(S'',\tau'')$ are \textit{weakly equivalent} if there exists a Klein surface $(S^\ast,\tau^\ast)$ covering both
$(S',\tau')$ and $(S'',\tau'')$.

\vspace{1ex}

Consider the equipped family of real forms
$$
F=\{\hat{S}, G, (G_1,\hat{\tau}_1),\dots, (G_r,\hat{\tau}_r) \}
$$
corresponding to $\Omega$. The main topological invariants of the family are the genus $\hat{g}$ of $\hat{S}$,
the numbers of the ovals $\hat{k}_i= |\pi_0(\partial(\hat{S}/\hat{\tau}_i))|$ and the cardinalities $|G_i|$
of the subgroups. From the definitions it follow that $(S_i,\tau_i)=(\hat{S}/G_i, \hat{\tau}_i/G_i)$.
Therefore
$$
g_i\leq\frac{\hat{g}-1}{|G_i|}+1 \quad \text{and} \quad k_i\leq\hat{k}_i\,.
$$


Let us describe some properties of foams related with their genuses. These properties depend not only on
the topological properties of a foam $\Omega=(S,\Delta,\varphi)$, but also on a holomorphic structure on it.
Further on we assume that $\hat{g}>1$. The topological classification of the equipped families of real forms
for $\hat{g}\leq 1$ is more simple, but requires other methods.

\begin{corollary} Let $\Omega=(S,\Delta,\varphi)$
be a Klein foam such that $S$ contains the union $\coprod\limits_{i=1}^r S_i$ with all $(S_i,\tau_i)$ foam-different.
Then $\sum\limits_{i=1}^r k_i\leq 42(\hat{g}-1)$.
\end{corollary}

\begin{proof}
The involutions $\hat{\tau}_i$ corresponding to foam-different Klein surfaces are different.
On the other hand, the sum of of the numbers of all the ovals of all the real forms of a Riemann surface of genus $\hat{g}>1$
is at most $ 42(\hat{g}-1)$ \cite{Na6,Na4}. Thus $\sum\limits_{i=1}^r \hat{k}_i\leq 42(\hat{g}-1)$.
\end{proof}

\begin{corollary} Let $\Omega=(S,\Delta,\varphi)$
be a Klein foam such that $S$ contains the union $\coprod\limits_{i=1}^r S_i$ with all $(S_i,\tau_i)$ not weakly equivalent.
Then $r\leq2(\sqrt{\hat{g}}+1)$ for any $\hat{g}>1$ and $r\leq4$, if $\hat{g}$ is even.
\end{corollary}

\begin{proof}
The involutions $\hat{\tau}_i$ corresponding to not weakly equivalent Klein surfaces are not conjugate in $\hat{G}$.
On the other hand, the number of conjugation classes of involution on a Riemann surface of genus $\hat{g}$ is at most
$2(\sqrt{\hat{g}}+1)$ for any $\hat{g}>1$ \cite{Na5} and at most  $4$ if $\hat{g}$ is even \cite{GI}.
\end{proof}

\begin{corollary} Let $\Omega=(S,\Delta,\varphi)$
be a Klein foam such that $S$ contains the union $\coprod\limits_{i=1}^r S_i$ with all $(S_i,\tau_i)$ not weakly equivalent.
In addition assume that  either $r=3, 4$ or all the surfaces $\hat{S}/G_i$ are orientable.
Then $\sum\limits_{i=1}^r k_i\leq 2\hat{g}-2+2^{r-3}(9-r)<2\hat{g}+30$.
\end{corollary}

\begin{proof}
The corresponding bounds for the number of involutions on a Riemann surface are proved in \cite{Na4,Na7,Na8,Na9}.
\end{proof}

\bigskip
{\bf Acknowledgements} Partially supported by the grants NSh--5138.2014.1, RFBR--13-01-00755 (S.G.-Z.), RFBR--13-02-00457 (S.N.).

Keywords: foam, real forms.

\bigskip
{\bf Acknowledgements} Partially supported by the grants NSh--5138.2014.1, RFBR--13-01-00755 (S.G.-Z.), RFBR--13-02-00457 (S.N.).

Keywords: foam, real forms.

\end{document}